\documentclass[11pt]{amsart}
\usepackage{amscd,amssymb,amsthm, amscd, amsfonts, eucal, hyperref}
\usepackage[all]{xy}
\usepackage[dvips]{graphicx}
\usepackage{epsfig}

\theoremstyle{plain}
\newtheorem{thm}[subsection]{Theorem}
\newtheorem{lem}[subsection]{Lemma}
\newtheorem{prop}[subsection]{Proposition}
\newtheorem{cor}[subsection]{Corollary}

\theoremstyle{definition}
\newtheorem{rk}[subsection]{Remark}
\newtheorem{definition}[subsection]{Definition}
\newtheorem{ex}[subsection]{Example}

\numberwithin{equation}{section}
\newcommand{\M}{{\mathcal M}}

\newcommand{\A}{{\mathcal A}}

\newcommand{\T}{\mathbb{T}}

\newcommand{\CC}{{\mathcal C}}
\newcommand{\LL}{{\mathcal L}}

\newcommand{\G}{{\mathcal G}}
\newcommand{\I}{{\mathcal I}}

\newcommand{\X}{{\mathcal X}}
\newcommand{\Z}{\mathbb{Z}}

\newcommand{\R}{\mathcal{R}}
\newcommand{\C}{\mathbb{C}}

\newcommand{\PP}{\mathbb{P}}

\DeclareMathOperator{\Hom}{Hom}

\begin{document}

\title{Admissibility of local systems for some classes of line arrangements}
\dedicatory{Dedicated to the memory of Dinh Thi Anh Thu.}

\author{Nguyen Tat Thang}
\address{Institute of Mathematics,
 18 Hoang Quoc Viet road,
Cau Giay District, 10307 Hanoi, Vietnam.}

\email {ntthang@math.ac.vn}



\begin{abstract}Let $\A$ be a line arrangement in  the complex projective plane $\PP^2$. Denote by $M$ its complement and by $\M$ the set of points in $\A$ with multiplicity at least $3$. A rank one local system $\LL$ on $M$ is admissible if roughly speaking the dimension of the cohomology groups
$H^k(M,\mathcal{L})$ can be computed directly from the cohomology algebra $H^{*}(M,\Bbb{C})$. In this work, we give a sufficient condition for the admissibility of all rank one local systems on $M$.
\end{abstract}

\maketitle

\section{Introduction } \label{s0}

When $M$ is a hyperplane arrangement complement in projective space $\PP^n$,
 the notion of an {\it admissible} local system $\LL$ on $M$ is defined in terms of some
conditions on the residues of an associated logarithmic connection $\nabla(\alpha)$ on a good
compactification of $M$ (see \cite{ESV}, \cite{STV}, \cite{F}, \cite{LY} and \cite{DM}). This notion plays a key
role in the theory, since for such an admissible local system $\mathcal{L}$ on $M$ one has
\begin{align}\label{iso1}
\dim H^k(M, \LL) = \dim H^k(H^{*}(M, \Bbb{C}), \alpha\wedge)
\end{align}
for all $k\in \Bbb{N}$.

Let $\A$ be a line arrangement in the complex projective plane $\mathbb{P}^2$ and denote
its arrangement complement by $M$. For the case of line arrangements,
a good compactification is obtained by blowing-up points of multiplicity larger than $2$ in $\A$. This explains the simple definition of the admissibility given
below in Definition \ref{dl}.

 In a recent paper \cite{NR}, the authors introduced a class of line arrangements for which all rank one local systems on the complements are admissible. Namely, for each non-negative integer $k$, the line arrangement $\A$ is called to be of type $\CC_k$ if $k$ is the minimal number of lines in $\A$ containing all the points of multiplicity at least $3$. It is proved in \cite{NR} that

\begin{thm}\label{thm-NR}
Let $\A$ be a line arrangement in $\PP^2$. If $\A$ belongs to the class $\CC_k$ for
some $k \leq 2$, then any rank one local system $\LL$ on $M$ is admissible.
\end{thm}

The purpose of this paper is to improve the work in \cite{Th}, see Remark \ref{rm1} below. More precisely, we give a combinatoric condition on a line arrangement $\A$ for the admissibility of rank one local systems on its complement $M$. 

Let  $\M$ be the set of points in $\A$ with multiplicity larger than $2$. Two points $x, y\in \M$ are called {\it adjacent} if they belong to a line $H\in \A$ (see \cite{Th}). Suppose that $\A$ satisfies the following condition:

\medskip

(C) For each point $x\in \M$, there exist at most two lines $H_1, H_2\in \A$ such that $x\in H_1\cap H_2$ and $H_1, H_2$ contain all points in $\M$ which are adjacent to $x$.

\medskip

If all points of multiplicity $\geq 3$ are situated on a line the arrangement is a nodal affine arrangement, see \cite{CDP, Th2}. Theorem \ref{thm-NR} above shows that for such arrangements which are called of type $\CC_1$ all rank one local systems on the complement are admissible.

In \cite{Th} the author defines the notion of a {\it path of length $n$} to be the maximal sequence of distinct lines $H_1, H_2, \ldots, H_n$ where $x_i= H_i\cap H_{i+1}\in \M$ and $\# \{x\in H_i : x\in \M\}\geq 2.$  One admits that lines in $\A$ containing only double points also make a path (i.e. the path $\{H\}$).

If a point $x\in \M$ is not adjacent to any point in $\M$ and if $H$ is a line passing through $x$, we consider $\{H\}$ as a path and for all $H^{'}$ passing through $x$, one identifies the path $\{H^{'}\}$ with $\{H\}$. A path is called a {\it cycle} if $H_1\cap H_n\in \M$ with $n\geq 3$, otherwise it is called {\it open} (see \cite{Th}).

 Our first main result is the following.

\begin{thm}\label{mainthm}
Let $\A$ be a line arrangement in $\PP^2$ satisfying condition (C). Assume that $\A$ has at most one cycle. Then all rank one local systems on the complement $M$ of $\A$ are admissible.

In particular, the characteristic variety $\mathcal{V}_1(M)$ does not contain translated components and $\mathcal{V}_1(M)$ is determined by the poset $L(\A)$.
\end{thm}

In Section 2 we first make explicit the admissibility condition in the case of line
arrangements and recall the definition of characteristic varieties. Then we prove Theorem \ref{mainthm}. In the end of Section 2 we give example of a line arrangement where the results in \cite{Th} and \cite{NR} can not be applied while Theorem \ref{mainthm} is useful (Example \ref{ex1}).

In Section 3 we concentrate on arrangements having more than one cycle. The mains results in this section are Theorems \ref{mainthm2} and \ref{thm3} where we show that, under some additional assumptions, one still has the admissibility of all local systems. As an evidence, we give in Example \ref{ex2} an arrangement and a nonadmissible local system on its complement. Accordingly, Theorem \ref{mainthm} does not hold if there are more than one cycle, also Theorem \ref{thm3} is not true without the condition (1). That means our results are best possible.

In the last section, we will study the {\it multinets} and {\it resonance varieties}. We will prove that if the line arrangements $\A$ satisfies the condition (C) then it does not support any multinets, equivalently, these is not any {\it global resonance component} except all lines in $\A$ are concurrent.

\section{Admissible rank one local systems} \label{s1}

Let $\A=\{H_0,H_1,...,H_n\}$ be a line arrangement in $\PP^2$ and set $M=\PP^2 \setminus (H_0 \cup...\cup H_n)$. Let $\T(M)=\Hom(\pi_1(M),\C^*)$ be the character variety of 
$M$. This is an algebraic torus
$\T(M) \simeq (\C^*)^{n}$. Consider the exponential mapping
\begin{equation} 
\label{e1}
\exp :H^1(M,\C) \to H^1(M,\C^*)=\T(M)
\end{equation}
induced by the usual exponential function $\exp(2 \pi i -): \C \to \C^*$. 

Clearly one has 
$\exp(H^1(M,\C))=\T(M)$
and  $\exp(H^1(M,\Z))=\{1\}$. More precisely, a rank one local system $\LL\in \T(M)$ corresponds to the choice of some monodromy complex numbers
$\lambda _j \in \C^*$ for $0 \leq j \leq n$ such that $\lambda _0 ...\lambda _n=1$. A cohomology class
$\alpha \in H^1(M,\C)$ is given by
\begin{equation} 
\label{e2}
\alpha=\sum_{j=0,n}a_j\frac {df_j}{f_j}
\end{equation}
where the residues $a_j \in \C$ satisfy $\sum_{j=0,n}a_j=0$ and $f_j=0$ a linear equation for the line $H_j$. With this notation, one has
$\exp (\alpha)=\LL$ if and only if $\lambda _j =\exp(2\pi i a_j)$ for any $j=0,...,n$.

\begin{definition} \label{dl}{\rm 
A local system $\LL \in \T(M)$ as above is {\it admissible} if there is a cohomology class $\alpha \in H^1(M,\C)$ such that $\exp(\alpha)=\LL$, $a_j\notin \Z_{>0}$ for all $j$ and for all point $p \in H_0 \cup...\cup H_n$ of multiplicity at least 3 one has
$$a(p)=\sum_ja_j \notin \Z_{>0}.$$
Here the sum is over all $j$'s such that $p \in H_j$.
}
\end{definition}

For an admissible local system the isomorphism in (\ref{iso1}) were proved in \cite{ESV}, \cite{STV}.

\begin{definition}{\rm 
The {\it characteristic varieties} of $M$ are the jumping loci for the first cohomologies of $M$ with coefficients in rank one local systems:
$$\mathcal{V}^i_k(M)=\{\rho\in \mathbb{T}(M) : \dim H^i(M, \mathcal{L}_{\rho})\geq k\}.$$
When $i = 1$ we use the simple notation $\mathcal{V}_k(M) = \mathcal{V}^1_k (M)$.}
\end{definition}

Foundational results on the structure of the cohomology support loci for local
systems on quasi-projective algebraic varieties were obtained by Beauville \cite{Be}, Green
and Lazarsfeld \cite{GL}, Simpson \cite{S} (for the proper case), and Arapura \cite{A} (for the
quasi-projective case and first characteristic varieties $\mathcal{V}_1(M)$).

Before prove Theorem \ref{mainthm} we introduce the following notion which is a generalization of a path as defined in the Introduction.

\begin{definition}
A subset $\G:= \{H_i\}_{i\in I}\subset \A$ is called a {\it graph} if it satisfies the following conditions:
\begin{enumerate}
\item[(i)] For all $i\in I$ then $\# \{x\in H_i : x\in \M\}\geq 2$;
\item[(ii)] For all $i\in I$, there exists $j\in I\setminus \{i\}$ such that $H_i\cap H_j\in \M$;
\item[(iii)] For any two points $x, y\in \M$ where $x= H_{i_1}\cap H_{i_2}, y= H_{j_1}\cap H_{j_2}$
with some $i_1, i_2, j_1, j_2\in I$, there exists a path $\{H_{k_1}, \ldots, H_{k_m}\}, k_i\in I$ as defined in the Introduction  such that 
$$x= H_{k_1}\cap H_{k_2}, y= H_{k_{m-1}}\cap H_{k_m}.$$
\end{enumerate}
If $H$ contains only one point $x$ in $\M$ which is isolated  (i.e. $x$ is not adjacent to any point in $\M$), we also call $\{H\}$ a graph. If $H$ does not contain any point of $\M$ we also admits that $\{H\}$ is a graph.

The graph $\G$ is said to be {\it maximal} if there does not exist $H\in\A\setminus \G$ such that $H\cup \G$ is a graph.
For each graph $\mathcal{G}$, we denote by $\{x_i\}_{i\in I_{\mathcal{G}}}$ set of points in $\M$ which belong to all lines in the graph. One defines {\it zone} $Z(\mathcal{G})$ associated to $\mathcal{G}$ as follows:
$$Z(\mathcal{G})= \{H\in \A: \exists i\in I_{\mathcal{G}}, x_i\in H\}.$$

\end{definition}

\begin{rk} (i) Every path is a graph.

(ii) One deduces from the hypothesis (C) that the non-isolated point $x\in \M$ determines lines  which contain $x$ and points of $\M$ adjacent to $x$. Then, a graph $\G$ will be characterized by the set of intersection points $x_i\in \M$.
\end{rk}

\begin{lem}\label{lm1}
Let $\A$ be a line arrangement in $\PP^2$ satisfying the condition (C). Then, the set of zones of all maximal graphs makes a partition of $\A$. 
\end{lem}

\begin{proof}
It is obvious that
$$\A= \cup_{\G}Z(\G),$$
where $\G$ runs over all maximal graphs of $\A$. Now, let consider zones associated to two maximal graphs $\G_1$ and $\G_2$. Assume that there exists $H\in Z(\G_1)\cap Z(\G_2)$. That means there are $H_1\in \G_1$ and $H_2\in \G_2$ such that
$$x_1= H\cap H_1\in \M, x_2= H\cap H_2\in \M.$$
If $x_1\neq x_2$ then $H\cup \G_1$ and $H\cup \G_2$ are graphs. Since $\G_i$ is maximal, one obtains that $H\in \G_1$ and $H\in \G_2$. It implies that $\G_1 \cup \G_2$ is also a graph. Hence $\G_1 = \G_2$. If $x_1\equiv x_2$ then $H_1\cup \G_2$ is a graph. It deduces $H_1\in \G_2$. Similarly, one obtains again that $\G_1 = \G_2$. The proof is complete.
\end{proof}

\begin{proof}[Proof of Theorem \ref{mainthm}]
 Let $\mathcal{L}$ be a local system on $M$. In order to find a good cohomology class $\alpha$ for $\mathcal{L}$, we will shape, graph by graph, the positive integer residues $a(p)$ for all $p\in \M$.

Fix one line $H_0$ in the cycle in $\A$ if it exists and any line in $\A$ containing at least two points in $\M$, otherwise. We see from Definition \ref{dl} that the admissibility condition bases on the real part of the residues $a_H, H\in \A$. So instead of those complex residues, we may consider their real parts. It means we can assume that all residues $a_H, H\in \A$ are real numbers. Without loss of generality, we may assume $a_H\in [0, 1)$ for all $H\in \A\setminus \{H_0\}$ and $a_{H_0}= - \sum_{H\neq H_0}a_H$. Recall that for each $x\in \M$ we denote $a(x)= \sum_{H\in \A, x\in H}a_H$. Let $\mathcal{G}$ be any maximum graph.
\medskip

{\bf Case 1: $H_0\notin \mathcal{G}.$} We will correct $a_H, H\in \G$ such that $a(x_i)\notin \Z_{> 0}, i\in I_{\mathcal{G}}$ by several steps.
\medskip

{\bf Step 1:} Start with a line $H_1\in \mathcal{G}$ such that there is only one line $H_2$ in $\mathcal{G}$ having intersection in $\M$ with $H_1$. Such a line exists since there does not exist any cycle in $\mathcal{G}$. Let
$$a_1:= \max\{0, a(p): p\in H_1\cap \M\setminus H_2, a(p)\in \mathbb{Z}_{> 0}\}.$$
Here and below, if the maximum is positive and attains at several points, we will take $a_1$ as the sum $a(p)$ at the point $p$ which is not the intersection point of two lines in $\mathcal{G}$.

In this step, we replace $a_{H_1}$ by $a_{H_1}- a_1$ and $a_{H_0}$ by $a_{H_0}+ a_1.$ It is obviously insures that
$$\sum_{H\in \A}a_H=0$$
and $a(x)\notin \mathbb{Z}_{> 0}$ for all $x\in H_1\cap \M\setminus H_2$. Since $a_1$ is either $0$ or the sum of residues of some distinct lines $H\in Z(\mathcal{G})$ with $H\neq H_0$ one still has $a_{H_0}\leq 0.$

{\bf Step 2:} We continue with the line $H_2$ defined in Step 1. Let
$$a_2:= \max \{0, a(p): p\in H_2\cap \M\setminus (\cup_{H\in \G\setminus \{H_1, H_2\}} H), a(p)\in \mathbb{Z}_{> 0}\}.$$
Denote by $H_2^j, j\in J$ lines in $\mathcal{G}$ satisfying the following conditions:
\medskip 

\centerline{$H_2^j\neq H_1, p_2^j:= H^j_2\cap H_2\in \M, a(p_2^j)\in \mathbb{Z}_{> 0}$ and $a_2< a(p_2^j), \forall j\in J$.}
\medskip 
We consider the following three possibilities.

{\bf (a)} $\# J\geq 2$ and $a_H=0$ for all $H\in \A\setminus \mathcal{G}$ passing through some $p_2^j$: One sees that 
$$a(p_2^j)= a_{H_2}+ a_{H^j_2}\in [0, 2).$$
 It implies $a(p_2^j)=1$ and hence $a(p_1)\notin \mathbb{Z}_{> 0}$ where $p_1= H_1\cap H_2.$ 
Then, we repeat the process from the beginning using the same method as in Step 1 for the maximal graph in $\mathcal{G}\setminus \{H_1\}$ which contains $H_2$ (in this case, it is actually $\mathcal{G}\setminus \{H_1\}$). 

{\bf (b)} $\# J\geq 2$ and there exist $H_2^{'}\in \A\setminus \mathcal{G}, j_0\in J$ such that $p_2^{j_0}\in H^{'}_2$ and $a_{H_2^{'}}\neq 0:$ Let
$$a_2^{'}:= \max\{a(p): p\in H_2\cap \M, a(p)\in \mathbb{Z}_{> 0}\}.$$
We replace $a_{H_2}$ by $a_{H_2}- a_2^{'}$ and $a_{H_2^{'}}$ by $a_{H_2^{'}}+ a_2^{'}.$ Note that this does not change $a(p_2^{j_0})$ but $a(p)\notin \mathbb{Z}_{> 0}$ for all $p\in H_2\cap \M\setminus H_2^{j_0}.$ Since $a_{H^{'}_2}\in (0, 1)$ one still has $a_H\notin \Z_{> 0}$ for all $H\in \A$.

In the next step, we continue with $H_2^j, j\in J$ simultaneously. For each $H_2^j$ we use the same method as we do with $H_2$.

{\bf (c)} $\#J\leq 1$: In this case, we correct residues as follows:
\medskip

\centerline{ $a_{H_2}:= a_{H_2}- a_2$ and $a_{H_0}:= a_{H_0}+ a_2.$}
\medskip
 It is easy to verify that $a(p)\notin \mathbb{Z}_{> 0}$ for all $p\in H_2\cap \M\setminus (\cup_{i\in J}H_2^j)$.

Similarly, in the next step  we process with $H_2^j, j\in J$ simultaneously. For each one, we repeat the method as in Step 2.

We continue the process until the residue of all lines in the graph are corrected (may be changed or not). By the method we replace residues one can easily see that: 
\medskip

{\bf Claim 1:} {\it $a(p)\notin \mathbb{Z}_{> 0}, \forall p\in \cup_{H\in \mathcal{G}} H\cap \M$.}

\medskip

{\bf Claim 2:} {\it $\sum_{H\in \A} a_H= 0.$}

\medskip

In each step, we add to $a_{H_0}$ integer numbers which are either $0$ or positive. In case of positive numbers, each of them has the form as follows:
$$a(p)= \sum_{H\in \A, p\in H}a_H, p\in \M.$$
For $H\in \G$ we denote by $b_H$ the origin residue of $H$ (i.e. before replacements) and for $x\in \M$ denote by $b(x)$ the sum $\sum_{x\in H}b_H$. We shall prove the following.
\medskip

{\bf Claim 3:} {\it The sum $A(\mathcal{G})$ we added to $a_{H_0}$ after correcting residue of all lines in $\mathcal{G}$ is
$$A(\G)= \sum_H b_H,$$
where $H$ runs over some distinct lines in $Z(\mathcal{G})$. Consequently, one has $a_{H_0}\leq 0$.}

\medskip

Before proving Claim 3, we consider what we added to $a_{H_0}$ after first two steps: 
$$A_2:= a_1+a_2.$$
 If $a_1=0$ then $A_2=a_2$ is either $0$ or 
$$A_2= \sum_{H\in Z(\G), x\in H} b_H,$$
for some $x\in \M$. Otherwise, if $a_2= a(p_1),$ where $p_1=H_1 \cap H_2$ then $$a_2= (b_{H_1}- a_1) + \sum_{H\in Z(\G), H\neq H_1, p_1\in H}b_H.$$
Therefore 
$$A_2= \sum_{H\in Z(\G), p\in H}b_H.$$
If $a_2= a(q)$ for some $q\in \M\setminus H_1$, it is easy to see the similar property of $A_2$.

In order to show Claim 3 we write $A(\G)$ as follows:
$$A(\G)= \sum_{i=1}^m \sum_{j=1}^{s_i}a_{i, j}$$
where $a_{i, j}\geq 0, j=1, \ldots, s_i$ denote the integer numbers we added to $a_{H_0}$ in Step $i$ when we correct residues of $H_i^j$ (we rename lines whose residues were corrected in Step $i$ by $H_i^j$) and $m$ is the number of steps we correct residues of all lines in $\G$. 

One reminds that 
$$a_{i, j}= \max \{0, a(x): x\in H_i^j\cap \M\setminus (\cup_{H\in \G\setminus (\cup_{k}H^k_{i-1})\cup H_i^j}H), a(x)\in \Z_{> 0}\}.$$
It means $a_{i, j}$ is either $0$ or 
$$a_{i, j}= \sum_{H\in Z(\G), x\in H} b_H= b(y)$$
with some $y\in H_i^j\cap \M\setminus (\cup_{H\in\G\setminus \{H_i^j\}}H)$ or
$$a_{i, j}= (b_{H_{i-1}^l}- a_{i-1, l})+ \sum_{H\in Z(\G), H\neq H_{i-1}^l, y\in H} b_H$$
where $H_{i-1}^l\in \G$ has intersection in $\M$ with $H_i^j$ and $y= H_{i-1}^l\cap H_i^j\in \M$.  In the last case, we see that $a_{i, j}+ a_{i-1, l}= b(y)$. Note that once we have $a_{i, j}$ in the last form, we also have the associated $a_{i-1, l}$ as a term of $A(\G)$. The corresponding between those terms is one-to-one due to the way of correcting residues. 

 Now, we pair terms in $A(\G)$ as follows: We start with $a_{m, j}, j=1,\ldots, m$. If $a_{m, j}$ has the last form, we pair it with the associated $a_{m-1, l}$, unless we keep it alone. By the same way, we continue with $a_{m-1, j}$ which is not in a pair. We repeat the process until each of $a_{i, j}$'s is either in a pair or has one of the first two forms as above. 
Finally, one obtains that
$$A(\G)= \sum b(y),$$
where the sum takes over some distinct points $y$ of  $\cup_{H\in \G}(H\cap \M)$. It is easy to check that there do not exist  two points in $y$'s belonging to the same line in $\G$. Claim 3 is proved.

Our last claim is the following:
\medskip

{\bf Claim 4:} {\it After replacing all residues we get $B(\G):= \sum_{H\in Z(\mathcal{G})} a_H\geq 0$.}

\medskip

It is the consequence of Claim 3 and the fact $B(\G)= \sum_{H\in Z(\G)} b_H- A(\G)$. Note that  $b_H\geq 0$ for all $H\in\G$.

{\bf Case 2:} $H_0\in \G$. We correct residues $a_H$ of all lines in the graph $\G\setminus \{H_0\}$ (we may choose $H_0$ such that $\G\setminus \{H_0\}$ is a graph) by the same way as in Case 1. By the same argument as above, we also receive properties as in the Claim (1-4).
  
 If $x\in \M$ is isolated then we replace $a_H$ by $a_H- a(x)$ and $a_{H_0}$ by $a_{H_0}+ a(x)$, where $H$ is any line containing $x$.
  
 To complete the proof, we need to check that $a(x)\notin \Z_{> 0}$ for all $x\in H_0\cap \M$. Indeed, if $x$ is not adjacent to any point out of $H_0$, according to Lemma \ref{lm1}, we have
 $$a(x)= -\sum_{x\notin H}a_H= - (\sum_{\I} \sum_H a_H + \sum a_{H^{'}}+ \sum a_{H^{''}}),$$
where $\I$ runs over all graphs of $\A$ which do not contain $H_0$, for each $\I$ then $H$ runs over all lines in its zone; in the second term, $H^{'}$ runs over all lines in the zone of the graph $\G\setminus \{H_0\}$, where $\G$ is the graph containing $H_0$ and the last sum takes over all lines which do not contain any point of $\M$ or contain only one point of $\M$ and have intersection in $\M$ with $H_0$. Each sum is non-negative according to Claim 4. Thus $a(x)\notin \Z_{> 0}.$

If $x$ is adjacent to $y\in H_1^{'}\neq H_0.$ We have
 $$a(x)= -\sum_{x\notin H}a_H= a_{H_1^{'}}- (\sum_{\I} \sum_H a_H + \sum a_{H^{'}}+ \sum a_{H^{''}}).$$
Because $a_{H_1^{'}}< 1$ and the sums are non-negative we obtain $a(x)< 1.$

Finally, we obtain residues of all lines in $\A$ for $\LL$ satisfying all conditions in Definition \ref{dl}. In other words, the local system $\LL$ is admissible. Combining this with results in \cite{Di} we get the properties of the characteristic varieties as shown in the theorem.

\end{proof}

\begin{rk}\label{rm1} 
In \cite{Th}, the author introduced a combinatoric condition for a line arrangement $\A$ such that all rank one local systems on its complement $M$ are admissible.

\begin{thm} [see \cite{Th}]\label{thm-Dinh}
Let $\A$ be a line arrangement in $\PP^2$ satisfying condition (C). Assume that $\A$ has at most one cycle and on each line $H\in \A$, there exist at most two points in $\M$ adjacent to points in $\M\setminus H$. Then, all local systems on complement $M$ of $\A$ are admissible.
\end{thm}
\end{rk}

The following is an example of an arrangement in $\CC_3$ for which both Theorem \ref{thm-NR} and Theorem \ref{thm-Dinh} cannot be applied, yet our new Theorem \ref{mainthm} shows that all local systems are admissible.

\begin{ex}\label{ex1}
Let $\A$ be the arrangement in $\PP^2$ defined by $13$ lines: $L_0: z=0, L_1: x=0, L_2: y=0, L_3: x+3y=3z, L_4: 3y-x=3z, L_5:x+4y=2z, L_6: x- 2y=2z, L_7: x+y= 4z, L_8: 5y-3x=12z, L_9: 2x=z, L_{10}: y=9z, L_{11}: y-x = 7z, L_{12}: y-x= 2z$, see Figure 1, there are no parallel lines here. 

\begin{figure}[ht]
\begin{center}
\includegraphics[scale= 0.6]{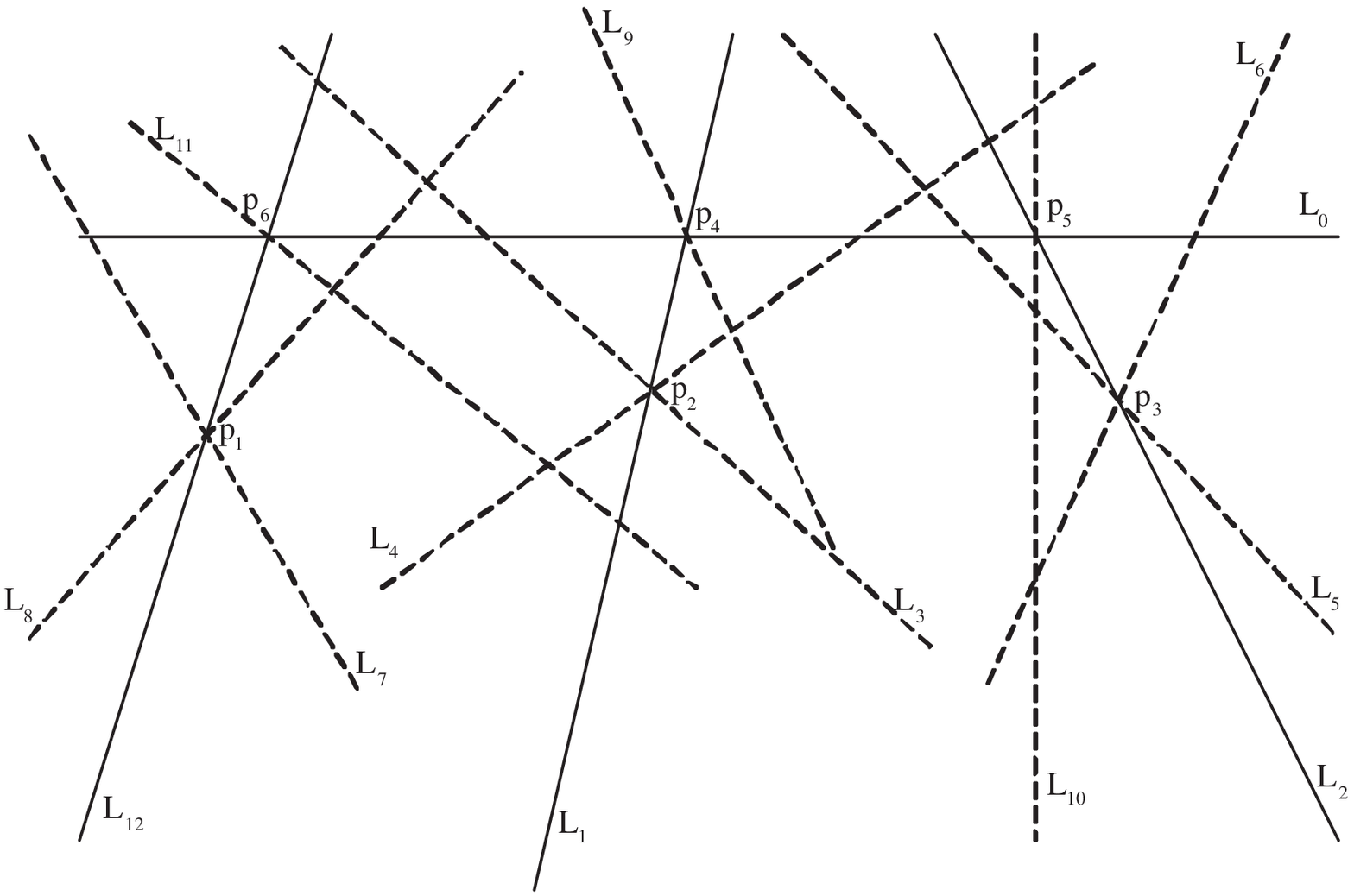}
\end{center}
\caption{}
\end{figure}

There are six points of multiplicity $3$, these are $p_1= [1:3:1]= L_{12}\cap L_7 \cap L_8, p_2= [0:1:1]= L_1\cap L_4\cap L_3, p_3= [2:0:1]= L_6\cap L_2\cap L_5, p_4= [0:1:0]= L_0\cap L_1\cap L_9, p_5= [1:0:0]= L_0\cap L_2\cap L_{10}, p_6= [1:1:0]= L_0\cap L_{11}\cap L_{12}$. 

Since there are $3$ points $p_4, p_5, p_6$ of $L_0$ which are adjacent to other points in $\M\setminus L_0$, the assumptions in Theorem \ref{thm-Dinh} are not all satisfied. Also, since $\A$ is of  type $\CC_3$ then Theorem  \ref{thm-NR} can not be applied. However, one can easily check that the condition (C) defined in the Introduction is fulfilled and there is not any cycle in $\A$. Therefore, according to Theorem \ref{mainthm}, all rank one local systems on the complement of $\A$ are admissible.

\end{ex}

\section{Admissibility for other classes of line arrangements}

In this section, we discuss the case where the arrangement has more than one cycle. One still has the admissibility of  local systems provided some certain assumptions.

\begin{thm}\label{mainthm2}
Let $\A$ be a line arrangement in $\PP^2$ satisfying condition (C). Assume that all cycles in $\A$ have at least one line in common. Then, all rank one local systems on the complement $M$ of $\A$ are admissible.
\end{thm}

\begin{proof}
We repeat the algorithm in the proof of Theorem \ref{mainthm} by firstly choosing $H_0$ to be the common line of all cycles in $\A$. The proof is then straightforward.
\end{proof}

Let $\A$ be a line arrangement in $\PP^2$ satisfying the condition (C). Denote by $\A_1$ set of all lines $H\in \A$ such that $H$ contains only one point of $\M$. Note that if $\M\neq\emptyset$ then $\A_1\neq\emptyset$ (unless there exists $x\in \M$ and there are at least three lines passing through $x$ which contain points adjacent to $x$, this contradicts to (C)). Let $\LL\in \T(M)$ be a rank one local system and choose residues $a_H$ as in Definition \ref{dl} for $H\in \A$. For each cycle $\CC= \{H_1, \ldots, H_s\}$ in $\A$, we denote by $P_{\CC}$ the following subset of $\M$:
$$P_{\CC}:=\{p_1, \ldots, p_s\},$$
where $p_j=  H_j\cap  H_{j+1}$ for $j=1, \ldots, s-1$ and $p_s=H_s\cap  H_1$.

\begin{prop}\label{prop1}
Let $\A$ be a line arrangement satisfying the condition (C) and fix a line $H_0$ in $\A$. Assume that for any cycle $\CC$ in $\A$ not involving the line $H_0$ there exists $H\in \A_1$ with $H\cap P_{\CC}\neq \emptyset$ such that $a_H\notin \Z$. Then, the local system $\LL$ is admissible.

\end{prop}

\begin{proof}
By the same argument as in the beginning of the proof of Theorem \ref{mainthm}, we may assume that $a_H\in [0, 1)$ for all $H\in \A\setminus \{H_0\}$. Then the condition $a_H\notin \Z$ means $a_H\neq 0$. The idea of the proof is the same as in proof of Theorem \ref{mainthm}, but firstly we open cycles in $\A$.

Let $\G$ be any graph of $\A$. For each cycle $\CC$ in $\G$ which does not contain $H_0$, according to the hypothesis, we can choose a line $H_{\CC}\in \A_1$ such that $a_{H_{\CC}}\in (0, 1)$ and $H_{\CC}$ passes through some point $p_{\CC}= H_{\CC}^1\cap H_{\CC}^2 \in P_{\CC}$, where $H_{\CC}^1, H_{\CC}^2\in \CC$. Let 
$$a:= \max \{0, a(x): x\in H_{\CC}^1\cap \M, a(x)\in \Z_{> 0}\},$$
where $a(x)= \sum_{H\in \A, x\in H}a_H$. Then, in the first step, we replace residues as follows:
$$a_{H_{\CC}^1}:= a_{H_{\CC}^1}- a, a_{H_{\CC}}:= a_{H_{\CC}}+ a.$$ 
After this replacement, we see that $a_{H_{\CC}}\notin \Z_{> 0}$ and $a(x)\notin \Z_{> 0}$ for all $x\in H_{\CC}^1\cap \M\setminus \{p_{\CC}\}$. However $a(p_{\CC})$ and $a_H$ with $H\notin \{H_{\CC}^1, H_{\CC}\}$ do not change.

Now we repeat the process as in the proof of Theorem \ref{mainthm} for each graph in $$\G^{'}:= \G\setminus (\cup_{\CC}\{H_{\CC}^1\}\cup H_0).$$
During the process, we regard $p_{\CC}$ as a point of $\G^{'}$ of multiplicity at least $3$ so that the corresponding residue $a(p_{\CC})$ is corrected when we shape the residue of $H_{\CC}^2$. 

Finally, we obtain new residues with all conditions as in Definition \ref{dl} satisfied.
\end{proof}

\begin{thm}\label{thm3}
Let $\A$ be a line arrangement satisfying the condition (C) and fix a line $H_0$ in $\A$. Assume that for any cycle $\CC$ in $\A$ not involving the line $H_0$ the followings hold:
\medskip

(1) The number of lines in $\CC$ is even;
\medskip

(2) On each line $H\in \CC$, there exist at most two points in $\M$ adjacent to other points in $\M\setminus H$.

Then, all rank one local systems on the complement $M$ of $\A$ is admissible.
\end{thm}

\begin{proof}
Let $\LL$ be any rank one local system on $M$ with residues $a_H, H\in \A$. Similarly, we may assume that $a_H\in [0, 1)$ for all $H\in \A, H\neq H_0$.

Let $\G$ be a maximal graph. If $\G$ contains $H_0$ or $\G$ does not contain any cycle, we will shape residues of line in $\G\setminus \{H_0\}$ by using the same method as in the proof of Theorem \ref{mainthm}. Since there is not any cycle in $\G\setminus \{H_0\}$ all Claims and argument there hold in this situation. Otherwise, according to the hypothesis, the graph $\G$ is itself a cycle which satisfies conditions (1) and (2) above, namely $\CC= \{H_1, \ldots, H_{2k}\}$. We consider the following possibilities.
\medskip

(i) There exists a point $p\in P_{\CC}$ such that $a(p)= \sum_{H\in \A, p\in H}a_H\notin \Z_{>0}$: Without loss generality, we can assume that $p= H_1\cap H_{2k}$. Then, we repeat the algorithm as in the proof of Theorem \ref{mainthm} for $\CC$ by firstly correcting the residue of $H_1$: put
$$a_1= \max \{0, a(x): x\in H_1\cap \M\setminus H_2, a(x)\in \Z_{> 0}\}.$$
In the first step, replace $a_{H_1}$ by $a_{H_1}-a_1$ and $a_{H_0}$ by $a_{H_0}+ a_1$. We continue the process with $H_2$ until residues of all lines are corrected.

(ii) There exists $H\in \A_1, H\cap P_{\CC}\neq \emptyset$ such that $a_H\neq 0$: Using same method as in proof of Proposition \ref{prop1}.
\medskip

(iii) For all $p\in P_{\CC}$ then $a(p)\in \Z_{>0}$ and for all $H\in \A_1, H\cap P_{\CC}\neq \emptyset$  we have $a_H=0$: In this case, due to $a_H\in [0, 1), H\in \CC$ then for $p\in P_{\CC}$ we obtain $a(p)\in [0, 2)$, hence $a(p)=1$. We replace residues as follows:
$$a_{H_{2i}}:=a_{H_{2i}}- a(p_{2i-1}), a_{H_{0}}:= a_{H_{0}}+a(p_{2i-1}), i=1, \ldots, k,$$
where $p_{2i-1}= H_{2i-1}\cap H_{2i}\in P_{\CC}$. It is easy to see that after those replacements all Claims as in proof of  Theorem \ref{mainthm} remain true.

Thus we get new residues for $\LL$ with all conditions as in Definition \ref{dl} satisfied. In other words $\LL$ is admissible.
\end{proof}

Let $\LL$ be a rank one local system on the complement $M$ of a line arrangement $\A$ and $\lambda_H\in \C^{*}$ for $H\in \A$ be the corresponding monodromy numbers as in Definition \ref{dl}. By the same argument as in the proof of Theorem \ref{thm3} above, one can show the following.

\begin{cor}
Let $\A$ be a line arrangement satisfying the condition (C) and fix a line $H_0$ in $\A$. Assume that for any cycle $\CC$ in $\A$ not involving the line $H_0$, on each line $H\in \CC$, there exist at most two points in $\M$ adjacent to other points in $\M\setminus H$. 

Then, either $\LL$ is admissible or there exists a cycle $\CC$ such that $\lambda_H= -1$ for all $H\in \CC$ and $\lambda_H=1$ for all $H\notin \CC$ which has intersection in $\M$ with some line of $\CC$.

\end{cor}

\begin{rk}

In the following example, we will see that among arrangements satisfying the condition (C) one can not remove the assumption in Theorem \ref{mainthm} as well as the condition (1) in Theorem \ref{thm3}. 
\end{rk}

\begin{ex}\label{ex2}
Let consider the arrangement $\A$ in $\PP^2$ consists of $12$ lines: $L_1: x=0, L_2: y=0, L_3: x+y-z=0, L_4: x+3y=0, L_5: x-3y-z=0, L_6: 3x-y+z=0, L_7: x-y+2z=0, L_8: 4x+y-12z=0, L_9: x+2y-10z=0, L_{10}: x-y+8z=0, L_{11}: 4x+y+12z=0$ and $L_0: z=0$, this last one is the line at infinity. There are $6$ points of multiplicity at least $3$: $p_1= [0:0:1]= L_1\cap L_2\cap L_4, p_2= [1:0:1]= L_2\cap L_3\cap L_5, p_3= [0:1:1]= L_1\cap L_3\cap L_6, p_4= [2:4:1]= L_7\cap L_8\cap L_9, p_5= [1:1:0]= L_0\cap L_7\cap L_{10}, p_6= [1:~-4:~0]= L_0\cap L_8\cap L_{11}$, see Figure 2, there are two disjoint cycles of length $3$, without any line in common.

We consider the rank one local system $\LL= \exp(\alpha)$, where the cohomology class $\alpha\in H^1(M, \C)$ is given by residues $a_i:= a_{L_i}= 1/2$ for $i\in\{1, 2, 3, 7, 8\}$, $a_j:= a_{L_j}= 0$ for $j\in \{4, 5, 6, 9, 10, 11\}$ and $a_0:= a_{L_0}= -5/2$. We will prove that this local system $\LL$ is not admissible.

\begin{figure}[ht]
\begin{center}
\includegraphics[scale= 0.6]{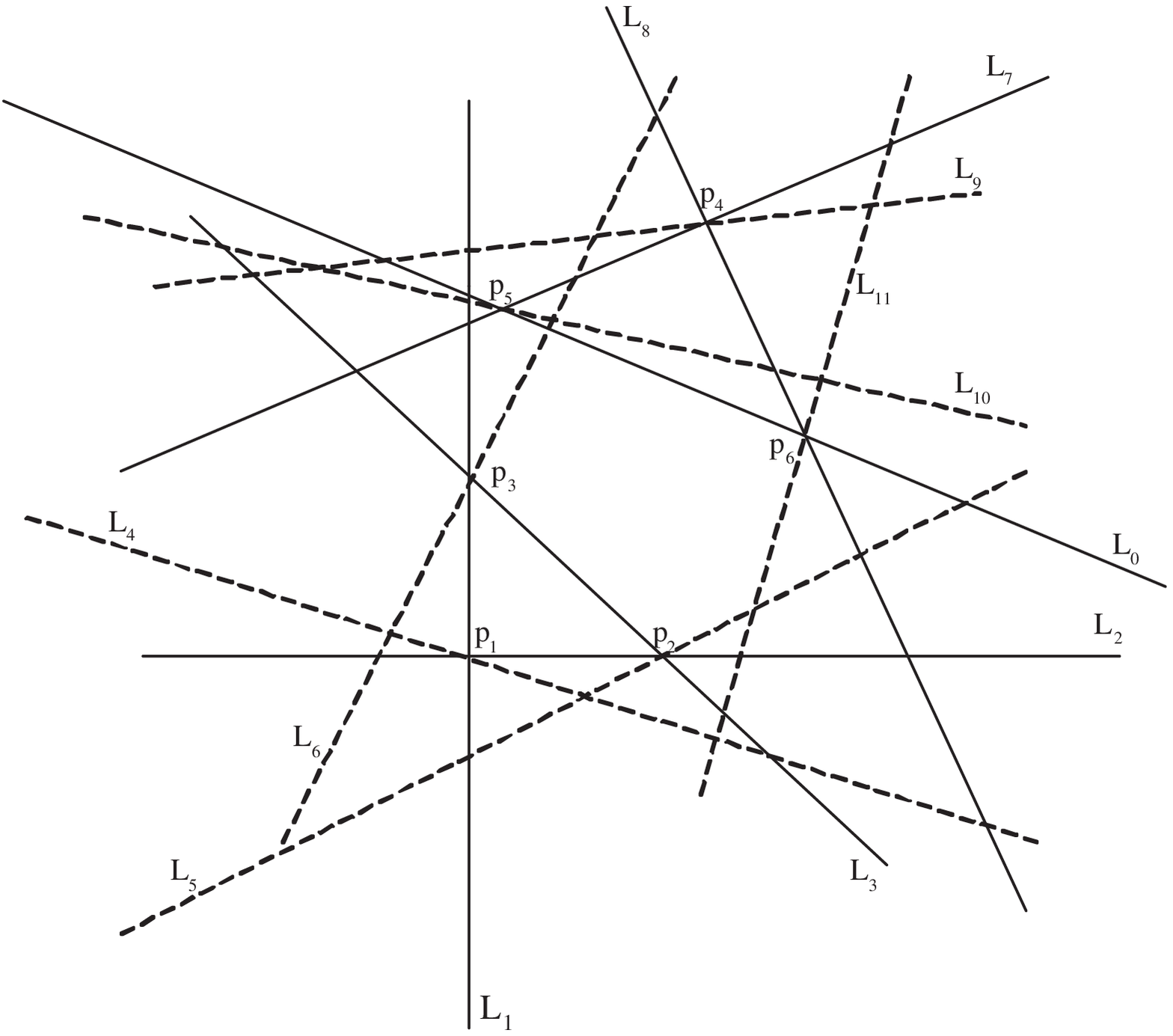}
\end{center}
\caption{}
\end{figure}

Indeed, assume by contradiction that $\LL$ is admissible. It means, there exists a cohomology class $\alpha^{'}\in H^1(M, \C)$ defined by residues $b_i:= b_{L_i}\in \C, i=0, \ldots, 11$ such that $\exp(\alpha^{'})= \LL, \sum_{i=0}^{11}b_i=0, b_i\notin \Z_{> 0}$ for any $i$ and $b(p_j)\notin \Z_{> 0}$ for any $j= 1, \ldots, 6$, where 
$$b(p_j)= \sum_{p_j\in L_k} b_k.$$
It is easy to see that $b_i= k_i+ 1/2$ for $i\in\{0, 1, 2, 3, 7, 8\}$ and $b_j= k_j$ for $j\in \{4, 5, 6, 9, 10, 11\}$ with $k_i\in \Z$ for all $i\in \{0, 1, \ldots, 11\}$. Since $b_i\notin \Z_{> 0}$ we get $k_j\leq 0$ for $j\in \{4, 5, 6, 9, 10, 11\}$.

We have the following equalities:
\begin{align*}
\sum_{i=1}^6b(p_i)&= 2\left(\sum_{i\in\{0, 1, 2, 3, 7, 8\}} b_i\right) + \sum_{j\in \{4, 5, 6, 9, 10, 11\}}b_j\\
                  &= - \sum_{j\in \{4, 5, 6, 9, 10, 11\}}k_j \geq 0.
\end{align*}
In other words $\sum_{i=1}^6b(p_i)\in \Z_{\geq 0}$. Moreover,  one observes that $b(p_i)$ is an integer for each $i= 1, \ldots, 6$. Therefore $b(p_i)=0$ for all $i$ (since $b(p_i)\notin \Z_{> 0}$) and hence $k_j=0$ for all $j\in \{4, 5, 6, 9, 10, 11\}$. In particular $b_1+ b_2= b_2+ b_3= b_1+ b_3=0$ so $b_1= b_2= b_3=0$ which is impossible.

Thus $\LL$ is not admissible.
\end{ex}

\section{Multinets and resonance varieties}

In this section, we will work on the resonance varieties concerning our line arrangements and discuss how these resonance varieties behave. We use  the notion of {\it multinets} which is defined in \cite{FY} where the authors gave the correspondence between the {\it global components} of the resonance varieties and the multinets.

\begin{definition}\label{def:multinet}(\cite{FY})
A {\it $(k, d)$-multinet} on a line arrangement $\A$ is a partition $\A= \cup_{i=1}^k(\A)_i$ of $\A$ 
into $k\geq 3$ subsets, together with an assignment 
of multiplicities, $m\colon \A\to \Z_{\geq 0}$, and a subset 
$\X\subset \M$ of multiple points, called the base locus, 
such that:
\begin{enumerate}
\item $\sum_{H\in\A_i} m_H=d$, independent of $i$;
\item For each $H\in\A_i$ and $H'\in \A_j$ with $i\neq j$, the point 
$H\cap H'$ belongs to $\X$;
\item For each $X\in\X$, the sum $n_X:=\sum_{H\in\A_i\colon H\leq X} m_H$
is independent of $i$;
\item For each $1\leq i\leq k$ and $H,H'\in\A_i$, there is a sequence
$H=H_0,H_1,\ldots, H_r=H'$ such that $H_{j-1}\cap H_j\not\in\X$ for
$1\leq j\leq r$.
\end{enumerate}
\end{definition}

\begin{lem}\label{lemMultinet}
Let $\A$ be  a line arrangements in $\PP^2$ such that the condition $C$ is satisfied. Then, either all lines in $\A$ are concurrent or $\A$ does not support any multinet. 
\end{lem}

\begin{proof}
Suppose that $\A$ supports a multinet $\A= \cup_{i=1}^k\A_i, k\ge 3$ with multiplicities $m\colon \A\to \Z_{\geq 0}$ and lines in $\A$ are not all concurrent. We denote by $\X$ the base locus. 

Let $H_1\in \A_1$ and $H_2\in \A_2$ arbitrary. According to the Condition (2) of Definition \ref{def:multinet} the point $p:= H_1\cap H_2\in \X$. If $p\in H$ for all $H\in \A_i, i>2$ there exists $H^{'}\in \A_1\cup \A_2$ such that $p\notin H^{'}$. Unless, there is at least one line $H\in \A_i$ for some $i>2$ where $p\notin H$. Anyway, there always exist at least $3$ lines belonging to different sets of $\A_i$'s which are not concurrent. Without any loss, we call them by $H_1\in \A_1, H_2\in\A_2, H_3\in\A_3$. According to the Condition(3) of Definition \ref{def:multinet}, number of lines in each $\A_i$ passing through each point of $\X$ are the same. Therefore, there exist $H_3^{'}\in\A$ which passes through the point $q_3:= H_1\cap H_2\in \X$ and $H_2^{'}\in \A_2$ passing through $q_2:= H_1\cap H_3\in \X$. But it implies from the Condition (2) that $H_2^{'}\cap H_3^{'}\in \X$. Hence the Condition (C) fails.
\end{proof}

The (first) {\it resonance varieties} of $\A$ are the jumping loci for the first cohomology of the complex  $H^*(H^*(M,\C), \alpha \wedge)$, precisely:
\begin{equation} \label{e4}
\R_k(\A)=\{\alpha \in H^1(M,\C) \mid \dim H^1(H^*(M,\mathbb{C}), \alpha \wedge)\ge k\}.
\end{equation}

It is proved in \cite{DPS2} that the irreducible components of resonance varieties are linear subspaces in $H^1(M,\C).$ A component $R$ of $\R_1(\A)$ is called a {\it global component} if $R$ is not contained in any coordinate hyperplane (see \cite{FY}).

\begin{thm}
Let $\A$ be  a line arrangements in $\PP^2$ such that the condition $C$ is satisfied. Then $\R_1(\A)$ does not contain any global component except all lines in $\A$ are concurrent.
\end{thm}

\begin{proof}
This Theorem is a corollary of Lemma \ref{lemMultinet} and the following fact.
\end{proof}

\begin{thm}[\cite{FY}] Suppose that the line arrangement $\A$ in $\PP^2$ supports a global resonance component of dimension $k-1$. Then $\A$ supports a $(k, d)$-multinet for some $d$.
\end{thm}

\section*{Acknowledgments} The author would like to thank professor Alexandru Dimca for useful discussions and comments. The author would like to thank professor Sergey Yuzvinsky for the valuable comment about multinets and resonance varieties.


\begin{thebibliography}{00}

\bibitem{A} D. Arapura: {\it Geometry of cohomology support loci for local systems, I,} J. Algebraic
Geom. {\bf 6}(1997), no. 3, 563-597.


\bibitem{Be} A. Beauville: {\it Annulation du $H^1$ pour les fibr\' es en droites plats,} in: Complex algebraic
varieties (Bayreuth, 1990), 1-15, Lecture Notes in Math., vol. 1507, Springer, Berlin, 1992.

\bibitem{CDP} A.D.R. Choudary, A. Dimca, S. Papadima: {\it Some analogs of Zariski's Theorem on nodal line arrangements,} Algebraic and Geometric Topol. {\bf 5}(2005), 691--711. 

\bibitem{Di} A. Dimca: {\it On admissible rank one local systems,} J. Algebra {\bf 321}(2009), 3145-3157. 

\bibitem{DM} A. Dimca, L. Maxim:  {\it Multivariable Alexander invariants of hypersurface complements,} 
Trans. Amer. Math. Soc. {\bf 359}(2007), no.~7, 3505 - 3528.

\bibitem{DPS2} A. Dimca, S. Papadima, A. Suciu: {\it Topology and geometry of coho-
mology jump loci,} Duke Math. J., {\bf 148} (2009), no. 3, 405-457.



\bibitem{Th} T. Dinh: {\it Arrangements de droites et syst\`emes locaux admissibles,} unpublished, PhD Thesis, Nice, 2009.

\bibitem{Th2} T. Dinh: {\it Characteristic varieties for a class of line arrangements,} Canad. Math. Bull. {\bf 54}(2011), no. 1, 56 - 67. 

\bibitem{ESV} H. Esnault, V. Schechtman, E. Viehweg: 
{\it Cohomology of local systems on the complement of hyperplanes,}
Invent. Math., {\bf 109} (1992), 557 - 561; Erratum, ibid. {\bf 112}(1993), 447.

\bibitem{F} M.~Falk:
{\it Arrangements and cohomology,}
Ann. Combin. {\bf 1}(1997), no.~2, 135 - 157.  

\bibitem{FY} M. Falk and S. Yuzvinsky: {\it Multinets, resonance varieties, and pencils of plane curves,} Compositio
Math. {\bf 143}(2007), 1069-1088.

\bibitem{GL} M. Green, R. Lazarsfeld: {\it Higher obstructions to deforming cohomology
groups of line bundles,} J. Amer. Math. Soc., {\bf 4}(1991), no. 1, 87- 103.


\bibitem{LY} A. Libgober, S. Yuzvinsky: {\it Cohomology of the Orlik-Solomon algebras and local systems,} Compositio Math. {\bf 121}(2000), 337 - 361.


\bibitem{NR} S. Nazir, Z. Raza : {\it Admissible local systems for a class of line arrangements,} Proc. Amer. Math. Soc. {\bf 137}(2009), no. 4, 1307 - 1313.


\bibitem{STV} V.~Schechtman, H.~Terao, A.~Varchenko:
{\it Local systems over complements of hyperplanes and the
{K}ac-{K}azhdan condition for singular vectors,}
J.~Pure Appl. Alg. {\bf 100}(1995),  no.~1-3, 93 - 102. 

\bibitem{S} C. Simpson: {\it Subspaces of moduli spaces of rank one local systems,} Ann. Sci. \'Ecole Norm.
Sup. {\bf 26}(1993), no. 3, 361- 401.


\end{thebibliography}
\end{document}